\documentclass{amsart}

\usepackage{amsfonts, amsmath, amsthm, amssymb}
\usepackage[all]{xy}

\newtheorem{theorem}{Theorem}[section]
\newtheorem{lemma}[theorem]{Lemma}
\newtheorem{cor}[theorem]{Corollary}

\newtheorem{definitiontemp}[theorem]{Definition}
\newenvironment{definition}{\begin{definitiontemp}
\normalfont}{\end{definitiontemp}}

\theoremstyle{remark}
\newtheorem*{remark}{Remark}

\newcommand{\cp}{\mathbb{C}_p}
\newcommand{\qp}{\mathbb{Q}_p}

\newcommand{\bnm}[2]{\genfrac{(}{)}{0pt}{}{#1}{#2}}
\newcommand{\D}[2]{{}^{#1}\! D^{#2}}
\newcommand{\Ds}[2]{{}^{#1}\! Q^{#2}}
\newcommand{\Di}[2]{{}^{#1}\! P^{#2}}
\newcommand{\Dz}{\mathcal{D}}

\newcommand{\h}{\mathcal{H}}

\title{Frobenius map for quintic threefolds}
\author{I. Shapiro}

\begin{document}

\maketitle

\begin{abstract}
We calculate the matrix of the Frobenius map on the middle dimensional cohomology of the one parameter family that is related by mirror symmetry to the family of all quintic threefolds.
\end{abstract}

\section{Introduction}
Recently $p$-adic methods were used to prove certain integrality results in the theory of topological strings (see \cite{ksv, sv1, sv2, vologodsky}). Namely, instanton numbers (i.e. genus zero Gopakumar-Vafa invariants) are defined in terms of the $A$-model and mirror symmetry allows us to express them in terms of the mirror $B$-model.  It follows from physical considerations that these numbers should be integers (they should coincide with the number of appropriate BPS states).  However integrality does not readily follow from either the $A$-model or the $B$-model interpretation.  The  integrality of instanton numbers and similar quantities (such as the number of holomorphic discs) were analyzed in \cite{ksv, sv1, sv2, vologodsky} by means of the $p$-adic $B$-model.  These numbers were expressed in terms of the Frobenius map on the middle dimensional $p$-adic cohomology; the integrality follows from this expression.

The main tool in the calculation of the Frobenius map is its relation with the Gauss-Manin connection.  This relation, however, does not determine the Frobenius map uniquely.\footnote{This is similar to the observation that a function is determined by a differential equation only up to boundary conditions.}  The additional  data required is the behavior of the Frobenius map at the boundary point of the moduli space (more specifically, the point of maximally unipotent monodromy that corresponds in the $B$-model to the infinite volume point in the $A$-model).  The analysis of the behavior of the Frobenius map at this point was carried out in \cite{vologodsky} using some very deep results in the theory of motives.  It was found that the matrix of the Frobenius operator (in a certain natural basis) at this point has at most one non-zero off-diagonal entry.\footnote{The diagonal entries range from $p^3$ down to $1$.}  This result was sufficient to prove the required integrality statements.

In the present paper we calculate the Frobenius matrix  at the boundary point of the moduli space, in the basic example of mirror symmetry, using the construction of the Frobenius map explained in \cite{homoint, local}.  This construction is equivalent to the original Dwork's construction of the Frobenius map.  Our calculations are very explicit and do not rely on any deep machinery from the theory of motives.  It follows from our computations that the natural conjecture that all the off-diagonal entries of the Frobenius matrix at the boundary point  are zero is false; in fact we obtain an explicit expression for the remaining undetermined entry and its non-vanishing can be verified by means of a computer calculation\footnote{The  calculations confirming non-vanishing have been performed by P. Dragon (using Mathematica) and the author (using PARI/GP).}.  More precisely, the remaining Frobenius matrix entry can be expressed as a $p$-adic series $$p^3\dfrac{24}{5^2}\left(\sum_{n= 3}^{\infty}\sum_{i=2}^{n-1}\sum_{j=1}^{i-1}\dfrac{B_n (n-1)!}{ij}-\dfrac{\left(\sum_{n\geq 1}B_n (n-1)!\right)^3}{6}\right)$$ in terms of the coefficients $B_n$ of the Dwork exponential $\sum B_n z^n:=\exp(z^p/p+z)$.

In fact it was pointed out to us by V. Vologodsky that from certain conjectures of the theory of motives one can derive that the above entry is a rational multiple of $\zeta_p(3)$, where the latter is a $p$-adic Riemann zeta value (see \cite{fur} for example).  We carried out (by computer) some calculations\footnote{We checked, to at least $10$ digits, the primes $3,5,7,11$ and $13$.  P. Dragon has independently confirmed these calculations.} that strongly point to the truth of the above.  Namely, the quantity in the parentheses above (later denoted by $\Delta_3$) is most naturally expressed in terms of the Kubota-Leopoldt $p$-adic $L$-function $L_p(s, \omega^{1-s})$ (where $\omega$ is the Teichm\"{u}ller character) that is related to the $p$-adic zeta values by $\zeta_p(s)=\frac{p^s}{p^s-1}L_p(s, \omega^{1-s})$ (see \cite{fur}).   Explicitly, it seems that $$\Delta_3=L_p(3, \omega^{1-3})/3$$ where we used the Dirichlet series expansion in \cite{del} to evaluate the latter. (This method of computing $p$-adic zeta values was suggested by A. Schwarz.)

\section{Preliminaries}

Recall that the $B$-model corresponding to the $A$-model on the quintic is a $1$-parameter family of mirror quintics defined as follows.  Consider the family given by the equations $$\lambda(x_0^5+x_1^5+x_2^5+x_3^5+x_4^5)+x_0 x_1 x_2 x_3 x_4=0$$ (let us denote these hypersurfaces by $V_\lambda$) inside the complex projective space $\mathbb{CP}^4$.  These should be factorized with respect to the symmetry group $\Gamma\cong (\mathbb{Z}/5\mathbb{Z})^3$. This group is realized as the quotient of the group of $5$-tuples of fifth roots of unity with product $1$ by the diagonal embedding of the fifth roots of unity.  The action is by multiplication of the coordinates by the corresponding roots of unity.  In principle one should then consider a resolution of the quotient; the resulting family $V^\circ_\lambda$ is referred to as the mirror quintics. It can be shown, however, (see for example \cite{coxkatz}) that for the purposes of computing the cohomology of $V^\circ_\lambda$ one may work with $\Gamma$ invariant elements in the cohomology of $V_\lambda$.  Observe that the permutation group on $5$ elements $\Sigma_5$ also acts on $V_\lambda$ by permutation of coordinates.  It is known that the $\Gamma$ invariant elements in the middle dimensional cohomology group of $V_\lambda$ are also $\Sigma_5$ invariant.

Note that the same constructions and statements remain true if one replaces $\mathbb{C}$ by $\mathbb{C}_p$ for $p\neq 5$, where $\cp$ denotes the completion of the algebraic closure of the
$p$-adic numbers $\qp$. Instanton numbers can be expressed in terms of the variation of Hodge structure on the $1$-dimensional family of mirror quintics.  The analysis of integrality of these instanton numbers performed in \cite{ksv} is based on the consideration of the Frobenius map.  If $p\neq 5$ then we may consider the invariant elements in the cohomology of $V_\lambda$ instead of working directly with the mirror quintics $V^\circ_\lambda$.

Thus we will be interested in computing the matrix of the Frobenius map on the $\Sigma_5\times\Gamma$-invariant part of the middle dimensional $p$-adic cohomology of $V_\lambda$, where the latter is now considered as a family of varieties over $\mathbb{C}_p$.  We will only be interested in the case of $\lambda$ small and in fact ``zero". More precisely, near $\lambda=0$, but not at $\lambda=0$ the cohomologies form a vector bundle.  It will turn out that there is a natural extension of all the structure that we consider, including the Frobenius map, to $\lambda=0$.

Denote by $\cp^\dagger\langle x_i\rangle$ the subring of the formal
power series $\cp[[x_i]]$ consisting of the overconvergent series.
More precisely, $\cp^\dagger\langle x_i\rangle$ consists of elements
$\sum a_I x^I$ with $ord_p a_I\geq c|I|+d$ with $c>0$, i.e. those
power series that converge on a neighborhood of the closed polydisc
of radius $1$ around $0\in\cp^n$.

For a collection of operators $D_i:A_i\rightarrow B$, we write $B/D_i$ to denote the quotient of  $B$ by span of the images of
the operators $D_i$, i.e. $B/D_i=\frac{B}{\sum_i D_i(A_i)}$.  Let $a^I$ stand for $a^{i+j+k+n+m+s}$, where $I=\{i,j,k,n,m,s\}$ is
a multi-index.  Denote by $\pi\in\cp$ an element such that $\pi^{p-1}=-p$.

In addition to $p\neq 5$, we will further assume that $p\neq 2$.

\section{General structures}
We review at this point some general facts about the cohomology of the
quintic.

\subsection{Definition} Consider the one parameter family $V_\lambda$ of projective
Calabi-Yau 3-folds over $\cp$ given by
$$\varphi_\lambda=\lambda(x_0^5+x_1^5+x_2^5+x_3^5+x_4^5)+x_0 x_1 x_2 x_3 x_4.$$ Recall that we are interested in the invariant part of the middle dimensional
cohomology of $V_\lambda$ in the neighborhood of the degeneracy
point $\lambda=0$. More precisely, we need the following
definitions.

\begin{definition}
Let $H_{dR}^3(V_\lambda)$ denote the $D$-module on the parameter
space obtained by computing the relative de Rham cohomology of the
family.
\end{definition}

In our parametrization of the family $V_\lambda$ we are interested in the neighborhood of $\lambda=0$.  Recall that outside of $\lambda=0$, the
family $H_{dR}^3(V_\lambda)$ is a  vector bundle with a flat
connection.  In fact the connection
has a regular (i.e., logarithmic) singularity at $\lambda=0$.  This prompts the following.

\begin{definition}
Let $H_{dR}^{3\times}(V_\lambda)$ denote the restriction of
$H_{dR}^{3}(V_\lambda)$ to the formal punctured disc, i.e.,
$$H_{dR}^{3\times}(V_\lambda)=H_{dR}^{3}(V_\lambda)\otimes_{\cp[\lambda]}\cp((\lambda)).$$
\end{definition}

As was mentioned previously, our interest is mainly in the invariant part of
$H_{dR}^{3\times}(V_\lambda)$.  More precisely, the $204$-dimensional
$H_{dR}^{3\times}(V_\lambda)$ has an action (see above)
of $\Sigma_5\times\Gamma$. We would like to
study only the 4-dimensional invariant part since that computes the cohomology of the mirror quintics $V^\circ_\lambda$.

\begin{definition}
Denote by $H_{dR}^{3\times}(V_\lambda)^{inv}$ the 4-dimensional
(over $\cp((\lambda))$) invariant part of
$H_{dR}^{3\times}(V_\lambda)$ under the action of $\Sigma_5\times\Gamma$.
\end{definition}

Below we will examine a certain important basis of
$H_{dR}^{3\times}(V_\lambda)^{inv}$.  It will turn out that in this
basis the Frobenius map extends to $\lambda=0$.  We will
compute as explicitly as possible the Frobenius matrix at
$\lambda=0$.

\subsection{Dwork cohomology}
As we are interested in doing this computation as directly as
possible we use Dwork cohomology.  We note that besides the shift in
cohomological degree which we immediately take into account, there
is the matter of the difference in the definition of Frobenius.
Namely, in our situation, the Frobenius defined through Dwork is
\emph{$p^2$ times the usual one}.

\begin{definition}
Let
$H_{dR}^{3\#}(V_\lambda)=H_{dR}^{3\times}(V_\lambda)\otimes_{\cp((\lambda))}\cp[[\lambda^{\pm
1}]]$ and correspondingly
$H_{dR}^{3\#}(V_\lambda)^{inv}=H_{dR}^{3\times}(V_\lambda)^{inv}\otimes_{\cp((\lambda))}\cp[[\lambda^{\pm
1}]].$
\end{definition}

The reason for considering $H_{dR}^{3\#}(V_\lambda)$ is that it can
be worked with very directly, in particular the Frobenius map is
very easily described on it, as well as the Gauss-Manin connection.  Concretely (see \cite{local} for more details), $H_{dR}^{3\#}(V_\lambda)$ is given as the top cohomology
of the relative\footnote{This means that
differentiation with respect to $\lambda$ is not used in the
differential, it is saved for the Gauss-Manin connection.} (with respect to $\lambda$),
overconvergent, homogeneous\footnote{The degrees are as follows:
$\text{deg}(\lambda)=0$, $\text{deg}(x_i)=1$ and $\text{deg}(t)=-5$.
}, twisted de Rham complex
$$DR(\cp^\dagger\left<x_0,...,x_4,t\right>[[\lambda^{\pm 1}]]e^{\pi
t\varphi_\lambda})_0$$ where the usual de Rham differential $d$ is
replaced by $$d+d(\pi t\varphi_\lambda)$$ and so
\begin{equation}\label{dworkeqn}
H_{dR}^{3\#}(V_\lambda)\cong \frac{(\cp^\dagger\langle x_0,...,x_4,t
\rangle[[\lambda^{\pm 1}]]\,dx dt)_0}{\genfrac{}{}{0pt}{}{(\partial_{x_i}+\pi
t(5\lambda x_i^4+x_0..\widehat{x}_i..x_4))dx_i}{(\partial_t+\pi
\varphi_\lambda)dt}}
\end{equation} with $dx=dx_0...dx_4$.

The action of $\Sigma_5$ permutes the variables $x_i$ and $\Gamma$ acts on the $x_i$ as before.

\begin{remark}
Observe that we have the obvious containments
$H_{dR}^{3\#}(V_\lambda)\supset H_{dR}^{3\times}(V_\lambda)$ and
$H_{dR}^{3\#}(V_\lambda)^{inv}\supset
H_{dR}^{3\times}(V_\lambda)^{inv}$.  While we are working in
$H_{dR}^{3\#}(V_\lambda)$, we will soon see that everything of
interest is happening inside $H_{dR}^{3\times}(V_\lambda)^{inv}$, in
fact in the $\cp[[\lambda]]$ span of our chosen basis.
\end{remark}

\subsection{The Frobenius map}
We describe an action of the Frobenius operator that acts on the parameter space
as well as on the fibers, i.e., near $\lambda=0$ only the ``fiber" at
$\lambda=0$ is preserved by the action.

Explicitly it is given by $$Fr:\omega(x,t,\lambda)\mapsto e^{\pi(t^p
\varphi_{\lambda^p}(x^p)-t\varphi_\lambda)}\omega(x^p,t^p,\lambda^p).$$
In our case $e^{\pi(t^p \varphi_{\lambda^p}(x^p)-t\varphi_\lambda)}$
decomposes as $$A(\lambda tx_0^5)...A(\lambda tx_4^5)A(xt)$$ where
$A(z):=e^{\pi(z^p-z)}$ with $A(z)=\sum A_i z^i$.

The definition of overconvergence is precisely formulated in such a
way that the cohomology of the overconvergent complex agrees with
the usual de Rham cohomology and at the same time it is possible to
define the Frobenius map as above, i.e., the function $A(z)$ is
overconvergent.

\subsection{The Gauss-Manin connection}
The Gauss-Manin connection plays an important r\^{o}le in the
computations below.  It can be described very explicitly on forms
and descends to cohomology. We have
$$\nabla_{\partial_{\lambda}}=\partial_\lambda+\partial_\lambda(\pi
t\varphi_\lambda)$$ and in fact we will be mostly interested in
$$\delta:=\lambda\nabla_{\partial_{\lambda}}$$ since the
Gauss-Manin connection has a logarithmic pole at $\lambda=0$.

The Frobenius map is compatible in a certain sense with the
connection. Namely, $$\delta\circ Fr=pFr\circ\delta.$$

\subsection{Symplectic structure}
The cohomology groups $H_{dR}^{3\times}(V_\lambda)$ also possess a
$\lambda$-linear non-degenerate symplectic form
$(-,-):H_{dR}^{3\times}(V_\lambda)^{\otimes 2}\rightarrow
\cp((\lambda))$. We are interested in its restriction to
$H_{dR}^{3\times}(V_\lambda)^{inv}$ which is still non-degenerate. The symplectic form is compatible with the Gauss-Manin connection in the usual sense, i.e.,
$$\delta(u,v)=(\delta u,v)+(u,\delta v)$$ as well as the Frobenius map, $$p^3 Fr(u,v)=(Fr u, Fr v).$$

\section{The classes $\omega_I$ and the cohomology near $\lambda=0$}
In this section we study the cohomology elements that appear in the
image of the Frobenius map.  We prove certain key recursion
relations that will allow us to make explicit computations later on.

\begin{definition}
Let $\omega$ be the image of $dx dt$ in $H_{dR}^{3\#}(V_\lambda)$;
it is in fact in $H_{dR}^{3\times}(V_\lambda)^{inv}$. For
$\lambda\neq 0$ it is the cohomology class of the nowhere vanishing
holomorphic $3$-form on the Calabi-Yau threefold $V_\lambda$.
\end{definition}

Since $\omega\in H_{dR}^{3\times}(V_\lambda)^{inv}$, and $\delta$ is
compatible with the $\Sigma_5\times\Gamma$ action, so that $\delta^i\omega\in
H_{dR}^{3\times}(V_\lambda)^{inv}$ for all $i$.

\begin{definition}
Let $\h$ denote the $\cp[[\lambda]]$ submodule of
$H_{dR}^{3\times}(V_\lambda)^{inv}$ spanned by
$\{\omega,\delta\omega,\delta^2\omega,\delta^3\omega\}$.  In fact
$\h\otimes_{\cp[[\lambda]]}\cp((\lambda))=H_{dR}^{3\times}(V_\lambda)^{inv}$
so $\h$ can be viewed as an extension (as a vector bundle) of
$H_{dR}^{3\times}(V_\lambda)^{inv}$ to $\lambda=0$.
\end{definition}

Our $\h$ is a vector bundle over the formal disc with a logarithmic
connection.  As we will show below, it is preserved by everything
that we consider.

\begin{lemma}
The $\cp[[\lambda]]$-module $\h$ is preserved by $\delta$.
\end{lemma}

\begin{proof}
This follows directly from Corollary \ref{pfcor}.
\end{proof}

\begin{definition}
Denote by $$\omega_{ijkmns}$$ the image of $$(\lambda
tx_0^5)^i(\lambda tx_1^5)^j(\lambda tx_2^5)^k(\lambda
tx_3^5)^m(\lambda tx_4^5)^n(tx)^s$$ in $H_{dR}^{3\#}(V_\lambda)$.
We write $x$ for $x_0...x_4$.
\end{definition}

\begin{lemma}
The elements $\omega_{ijknms}$ are in $\h$.
\end{lemma}
\begin{proof}
By Lemma \ref{s-rel} and Corollary \ref{i-rel} the elements
$\omega_{ijkmns}$ can be written as a linear combination of
$\delta^i\omega$  with constant (as functions of $\lambda$)
coefficients.  Thus $\omega_{ijkmns}$ can be written as a linear
combination of
$\{\omega,\delta\omega,\delta^2\omega,\delta^3\omega\}$ with power
series coefficients by Corollary \ref{pfcor}.
\end{proof}

We omit the indices that are $0$, unless it is the $s$ index. The
group $\Sigma_5$ acts on the $\omega_{ijknms}$ by permuting the
$i,j,k,n,m$, while $s$ remains fixed.  By the above Lemma, we don't
care where the non-zero, non-$s$ indices are, thus
$$\omega_{ijs}:=\omega_{ij000s}=\omega_{i00j0s},$$ etc.  Finally, we
sometimes write $\omega_I$ where $I$ is a multi-index, i.e.,
$$\omega_I=\omega_{ijknms}.$$

\begin{lemma}\label{frobandh}
The Frobenius map preserves $\h$.
\end{lemma}

\begin{proof}
Observe that $\delta^i\omega$ can be written as a linear combination
of $(xt)^j$, i.e., $\delta^i\omega=\sum_{j\leqslant i}a_j (xt)^j$
with $a_j$ constant.

Furthermore,\begin{align*} Fr((xt)^j dx dt)&=p^6 A(\lambda t
x_0^5)... A(\lambda t x_4^5) A(tx)(xt)^{jp+p-1} dx dt\\
&=p^6\sum A_i ... A_m A_s\omega_{ijknm(s+(j+1)p-1)}.
\end{align*}  Since $\omega_I$ are in $\h$ we are done.

\end{proof}

As a consequence of the Lemma above we see that we may consider the
restriction of the Frobenius map to $\lambda=0$, i.e., the matrix of
$Fr|_0:\h_0\rightarrow \h_0$.  An essential tool in this
investigation is the symplectic form $(-,-)$ that as we see below
also behaves well with respect to $\h$.

\begin{lemma}
The pairing $(-,-)$ maps $\h^{\otimes 2}$ to $\cp[[\lambda]]$ and
its restriction to $\lambda=0$, i.e.,
$$(-,-)_0:\h_0^{\otimes 2}\rightarrow \cp$$ is non-degenerate.
Furthermore, $(\omega,\delta^3\omega)_0=Y$ and so
$(\delta\omega,\delta^2\omega)_0=-Y$.\footnote{Since $\delta
f|_{\lambda=0}=0$ for $f$ a power series in $\lambda$, we have that
for $u$ and $v$ sections of $\h$, $(\delta u,v)_0=-(u,\delta v)_0.$
 The non-zero constant $Y$ is $Y(0)$, where $Y(z)$ is the Yukawa coupling, see Sec. \ref{pf}.}
\end{lemma}

\begin{proof}
For the first claim it is sufficient to check that
$(\delta^i\omega,\delta^j\omega)$ is a power series for arbitrary
$i,j\leq 3$.  Since
$(\delta^i\omega,\delta^j\omega)=\delta(\delta^{i-1}\omega,\delta^j\omega)-(\delta^{i-1}\omega,\delta^{j+1}\omega)$
it is enough to check that $(\omega,\delta^j\omega)$ is a power
series.  But this is true for $j\leqslant 3$ by Griffiths
transversality and Lemma \ref{yukawa}.  The pairing is
non-degenerate at $\lambda=0$ precisely because
$(\omega,\delta^3\omega)_0=Y\neq 0$ by Lemma \ref{yukawa} and
$(\delta\omega,\delta^2\omega)_0=-Y$ by the compatibility of the
symplectic pairing with the Gauss-Manin connection.
\end{proof}

\subsection{Differentiating $\omega_I$ and other relations}
Since the image of $\omega$ under the Frobenius map is expressed in terms of $\omega_I$, so it is necessary to re-express these elements in terms of our chosen basis of $\h$.  To that end we provide some key reduction formulas that allow us to accomplish that goal.

\begin{lemma}\label{s-rel}
We have the following $s$-relations:
\begin{equation}\label{sreduction}\delta\omega_{ijknms}=-(s+1)\omega_{ijknms}-\pi\omega_{ijknm(s+1)}\end{equation}
for $i,j,k,n,m,s\geq 0$.
\end{lemma}
\begin{proof}
Recall that
\begin{align*}
\delta&=\lambda\partial_\lambda+\pi t\lambda(x_0^5+...+x_4^5)\\
&=\lambda\partial_\lambda+\pi t(\varphi_\lambda-x)\\
\intertext{and by the last relation in Equation \ref{dworkeqn}}
&=\lambda\partial_\lambda-\partial_t t-\pi t x
\end{align*} so that
\begin{align*}
\delta\omega_{ijknms}&=(\lambda\partial_\lambda-\partial_t t-\pi t
x)((\lambda tx_0^5)^i...(\lambda tx_4^5)^m(xt)^s)\\
&=(i+...+m)\omega_I-(i+...+m+s+1)\omega_I-\pi\omega_{ijknm(s+1)}\\
&=-(s+1)\omega_{ijknms}-\pi\omega_{ijknm(s+1)}
\end{align*}

\end{proof}

\begin{lemma}\label{i-rel-one}
We have the following $i$-relations (by symmetry of the first five
relations in the Equation \ref{dworkeqn} they apply also to any
other index except $s$):
$$\omega_{(i+1)jknms}=-\frac{5i+(s+1)}{5\pi}\omega_{ijknms}-\frac{1}{5}\omega_{ijknm(s+1)}$$ for $i,j,k,n,m,s\geq 0$.
\end{lemma}
\begin{proof}
Recall that
\begin{align*}
\omega_{(i+1)jknms}&=(\lambda tx_0^5)^{i+1}...(\lambda
tx_4^5)^m(xt)^s\\
&=\lambda t x_0^4 x_0\omega_I\\
\intertext{and by the first relation in Equation \ref{dworkeqn}}
&=-\frac{1}{5\pi}(\partial_{x_0}x_0+\pi tx)\omega_I\\
&=-\frac{1}{5\pi}\left((5i+s+1)\omega_I+\pi\omega_{ijknm(s+1)}\right)
\end{align*}

\end{proof}

\begin{remark}\label{negative-i}
In fact, Lemma \ref{i-rel-one} holds for $i=-1$ and $s=4$ if we
correctly interpret it.  Namely,
$$\omega_{-\!100005}=\dfrac{(xt)^5}{\lambda
tx_0^5}=\dfrac{1}{\lambda^5}\omega_{011110}.$$  More precisely,
\begin{align*}
\omega_4&=(xt)^4\\
&=\lambda tx_0^4\dfrac{x_1^4...x_4^4
t^3}{\lambda}\\
&=-\dfrac{1}{5\pi}\left(\partial_{x_0}+\pi t
x_1...x_4\right)\dfrac{x_1^4...x_4^4
t^3}{\lambda}\\
&=-\dfrac{1}{5\lambda}x_1^5...x_4^5 t^4
=-\dfrac{1}{5\lambda^5}(\lambda tx_1^5)...(\lambda tx_4^5)
=-\dfrac{1}{5\lambda^5}\,\omega_{011110}.
\end{align*}  We will need this to compute the Picard-Fuchs equation
for $\omega$.

\end{remark}

\begin{cor}\label{i-rel}
A better version of $i$-relations (by symmetry of the first five
relations in the Equation \ref{dworkeqn} they apply also to any
other index except $s$):
$$\delta\omega_{ijknms}=5i\omega_{ijknms}+5\pi\omega_{(i+1)jknms}$$ for $i,j,k,n,m,s\geq 0$.
\end{cor}
\begin{proof}
Put Lemmas \ref{s-rel} and \ref{i-rel-one} together.
\end{proof}

\subsection{The Picard-Fuchs equation and the Yukawa coupling}\label{pf}
In this section we focus on the derivation of the Picard-Fuchs
equation for $\omega=\omega_{000000}$.  This is a fourth order
differential equation of the form
$$\delta^4\omega=g_3(\lambda)\delta^3\omega+g_2(\lambda)\delta^2\omega+g_1(\lambda)\delta\omega+
g_0(\lambda)\omega.$$  This equation is central to the study of the
quintic family, however the only property of it that we will use is
the fact that the coefficients $g_i$ are power series in $\lambda$
that vanish at $\lambda=0$.

\begin{definition}
Let $\#(I)=\#(\{i,j,k,n,m,s\})$ denote the number of non-zero
indices among $i,j,k,n,m$, i.e., wether or not $s=0$ does not affect
$\#(I)$.
\end{definition}

\begin{lemma}\label{pfequation}
Let $g(\lambda)=\dfrac{5^5\lambda^5}{5^5\lambda^5+1}$, then the
Picard-Fuchs equation for $\omega$ is
$$\delta^4\omega=-g(\lambda)(10\delta^3\omega+35\delta^2\omega+50\delta\omega+
24\omega).$$

\end{lemma}

\begin{proof}
The idea\footnote{Our approach is similar to \cite{candelas}.} is to
rewrite $\omega_{111100}$ as a linear combination of $\omega_{i\leq
4}$ using Lemma \ref{i-rel-one} and then express $\omega_{111100}$
in terms of $\omega_4$ as in the Remark following that Lemma.  Then
rewrite everything in terms of $\delta^{i\leq 4}\omega$ using Lemma
\ref{s-rel}.

Let $\widetilde{\omega}_{I}=(-5)^{\#(I)}\pi^I\omega_I$ then we see
that
\begin{align*}
\widetilde{\omega}_{1jknms}&=(-5)^{\#(I)}\pi^I\omega_{1jknms}\\
&=(-5)^{\#(I)}\pi^I\left(-\dfrac{s+1}{5\pi}\omega_{0jknms}-\dfrac{1}{5}\omega_{0jknm(s+1)}\right)\\
&=(s+1)(-5)^{\#(I)-1}\pi^{I-1}\omega_{0jknms}+(-5)^{\#(I)-1}\pi^{I}\omega_{0jknm(s+1)}\\
&=(s+1)\widetilde{\omega}_{0jknms}+\widetilde{\omega}_{0jknm(s+1)}.
\end{align*}  If we let $\widehat{\delta}=-\delta$ and
$\widehat{\omega}_s=\pi^s\omega_s$ then
$$
\widehat{\delta}\widehat{\omega}_s=-\delta\pi^s\omega_s
=(s+1)\pi^s\omega_s+\pi^{s+1}\omega_{s+1}
=(s+1)\widehat{\omega}_s+\widehat{\omega}_{s+1}. $$  Thus
$\widehat{\delta}^4\widehat{\omega}_0=\widetilde{\omega}_{111100}$,
i.e.,
$$\delta^4\omega=5^4\pi^4\omega_{111100}=-5^5\pi^4\lambda^5\omega_4.$$

Now by Lemma \ref{s-rel},
$\delta^4\omega=\alpha_4\omega_4+\alpha_3\omega_3+\alpha_2\omega_2+\alpha_1\omega_1+\alpha_0\omega_0$
with $\alpha_i$ constant; and it is clear that $\alpha_4=\pi^4$.  So
that
$$\delta^4\omega=g(\lambda)(\alpha_3\omega_3+\alpha_2\omega_2+\alpha_1\omega_1+\alpha_0\omega_0)$$
and again by Lemma \ref{s-rel},
$$\alpha_3\omega_3+\alpha_2\omega_2+\alpha_1\omega_1+\alpha_0\omega_0=a_3\delta^3\omega+a_2\delta^2
\omega+a_1\delta\omega+a_0\omega$$ with $a_i$ constant.

At this point the fact we need, namely that the coefficients of the
Picard-Fuchs equation are power series in $\lambda$ that vanish at
$\lambda=0$ is proved.  However to compute the $a_i$ we need to work
a bit more.

We note that \begin{align*} \omega&=\omega_0\\
\delta\omega&=-\omega_0-\pi\omega_1\\
\delta^2\omega&=\omega_0+3\pi\omega_1+\pi^2\omega_2\\
\delta^3\omega&=-\omega_0-7\pi\omega_1-6\pi^2\omega_2-\pi^3\omega_3\\
\delta^4\omega&=\omega_0+15\pi\omega_1+25\pi^2\omega_2+10\pi^3\omega_3+\pi^4\omega_4\\
\end{align*} and so we need to solve $$a_3\delta^3\omega+a_2\delta^2
\omega+a_1\delta\omega+a_0\omega=10\pi^3\omega_3+25\pi^2\omega_2+15\pi\omega_1+\omega_0$$
which yields $a_3=-10$, $a_2=-35$, $a_1=-50$ and $a_0=-24$.
\end{proof}

By induction we may now write any $\delta^{i>3}\omega$ as a linear
combination of $\delta^{i\leqslant 3}\omega$ with coefficients that
are power series in $\lambda$ that vanish at $\lambda=0$.  More
precisely, we have the following.

\begin{cor}\label{pfcor}
For $i>3$ we can write
$$\delta^i\omega=g_3^i(\lambda)\delta^3\omega+g^i_2(\lambda)\delta^2\omega+g^i_1(\lambda)\delta\omega+
g^i_0(\lambda)\omega$$ with $g^i_j$ power series in $\lambda$ that
vanish at $\lambda=0$, i.e., $$\delta^{i>3}\omega=0\,\, in\,\, \h_0.$$
\end{cor}
\begin{proof}
Proceed by induction.  It is true for $i=4$ by Lemma
\ref{pfequation}.  Now $\delta^{i+1}\omega=\delta(\sum_{j\leqslant
3}g^i_j\delta^j\omega)=\sum_{j\leqslant 3}(\delta
g^i_j)\delta^j\omega+\sum_{j<
3}g^i_j\delta^{j+1}\omega+g^i_3\delta^4\omega$.
\end{proof}

Now we can look at the symplectic pairing $(-,-)$.  Because it is
compatible with the Gauss-Manin connection we need only consider
pairings of the form $(\omega,\delta^i\omega)$.  By Griffiths
transversality $(\omega,\delta^i\omega)=0$ for $i\leq 2$.

\begin{definition}
Let $$Y(\lambda)=(\omega,\delta^3\omega)$$ denote the Yukawa
coupling.\footnote{A good reference for our treatment of the Yukawa
coupling is \cite{coxkatz}.}
\end{definition}

At this point we know only that $Y(\lambda)\in\cp((\lambda))$, but
in fact we can say more.  For example, $(-,-)$ is non-degenerate on
the span of $\{\delta^i\omega|0\leq i\leq 3\}$ and so
$Y(\lambda)\neq 0$. The following Lemma demonstrates that we can
calculate it up to a constant.  In fact, for our purposes it is
sufficient to prove that $Y(\lambda)$ is a power series that does
not vanish at $\lambda=0$.  For this we need only the fact that
$\lambda$ divides $g^4_3(\lambda)$.

\begin{lemma}\label{yukawa}
The Yukawa coupling $Y(\lambda)$ satisfies the differential equation
$$\partial_\lambda Y(\lambda)=-\dfrac{5^6\lambda^4}{5^5\lambda^5+1}Y(\lambda)$$ and so
$Y(\lambda)\in\cp[[\lambda]]^\times$, i.e. $Y(\lambda)$ is a power
series in $\lambda$ with a non-zero constant term.
\end{lemma}

\begin{proof}
As explained above $(\omega,\delta^2\omega)=0$, so that
$0=\delta^2(\omega,\delta^2\omega)=(\delta^2\omega,\delta^2\omega)+2(\delta\omega,\delta^3\omega)+(\omega,
\delta^4\omega)$ and
$$(\delta\omega,\delta^3\omega)=-\dfrac{1}{2}(\omega,
\delta^4\omega).$$  Then $$\delta
Y(\lambda)=(\delta\omega,\delta^3\omega)+(\omega,
\delta^4\omega)=\dfrac{1}{2}(\omega,
\delta^4\omega)=\dfrac{1}{2}(\omega,g(\lambda)(-10)\delta^3\omega)=-5g(\lambda)
Y(\lambda).$$
\end{proof}

\section{The coefficients $c^\alpha_I$ and the case of $\lambda=0$}
Since $\delta^{i\leqslant 3}\omega$ is a basis of $\h_0$ and
$(-,-)_0$ is non-degenerate, to compute the Frobenius matrix at
$\lambda=0$ it is sufficient to compute
$(Fr(\delta^i\omega),\delta^j\omega)_0$.  Because we can express
$\delta^i\omega$ in terms of $(xt)^j$ and $Fr((xt)^j dx dt)$, as we
have seen, is expressed in terms of $\omega_I$, it is important
to compute $(\omega_I,\delta^j\omega)_0$.  We begin this below.

\begin{definition}
Let $$c_I^\alpha=\frac{(-1)^I
\pi^I}{Y}(\omega_I,\delta^\alpha\omega)_0$$ which is non-trivial
only for $\alpha=0,1,2,3$ by Corollary \ref{pfcor}.
\end{definition}

We now translate the results of the previous section into relations
on $c^\alpha_I$ which are needed to compute the matrix elements of
$Fr$ at $\lambda=0$.

\begin{theorem}\label{c-rel}
The constants $c_I^\alpha$ are determined by symmetry in $i,j,k,n,m$
and the relations
\begin{equation}\label{ci-red}
c^\alpha_{(i+1)jknms}=i
c^\alpha_{ijknms}+\frac{1}{5}c^{\alpha+1}_{ijknms}
\end{equation}
and
\begin{equation}\label{cs-red}
c^\alpha_{ijknm(s+1)}=(s+1)c^\alpha_{ijknms}-c^{\alpha+1}_{ijknms}
\end{equation}
as well as the conditions that $$c_I^{\alpha>3}=0$$ and
$$c_{000000}^3=1.$$  Furthermore  the above holds for $c_I^\alpha$.
\end{theorem}

\begin{proof}
The symmetry of $c_I^\alpha$ follows from the symmetry of
$\omega_I$.  Notice that $c_I^{\alpha>3}$ vanishes since
$\delta^{\alpha>3}\omega$, when written as a linear combination of
$\delta^i\omega$ ($0\leq i\leq 3$), is divisible by $\lambda$.  By
definition $c_0^3=1/Y\cdot(\omega,\delta^3\omega)=1$.

Now we prove the two reduction formulas:
\begin{align*}
c^\alpha_{(i+1)jknms}&=\frac{(-1)^{I+1}\pi^{I+1}}{Y}(\omega_{(i+1)jknms},\delta^\alpha\omega)_0\\
&=\frac{(-1)^{I+1}\pi^{I+1}}{Y}(\frac{1}{5\pi}(\delta\omega_I-5i\omega_I),\delta^\alpha\omega)_0\\
&=\frac{1}{5}c^{\alpha+1}_I+i c^\alpha_I
\end{align*}
and
\begin{align*}
c^\alpha_{ijknm(s+1)}&=\frac{(-1)^{I+1}\pi^{I+1}}{Y}(\omega_{ijknm(s+1)},\delta^\alpha\omega)_0\\
&=\frac{(-1)^{I+1}\pi^{I+1}}{Y}(\frac{1}{\pi}(-(s+1)\omega_I-\delta\omega_I),\delta^\alpha\omega)_0\\
&=(s+1)c^\alpha_I-c_I^{\alpha+1}.
\end{align*}

The result now follows by induction.
\end{proof}

\begin{cor}\label{reductionconsequence}
The coefficient $c^\alpha_I$ with $\#(I)\geq 1$ can be written as a
linear combination of $c^{\alpha+1}_{I_j}$, i.e., $$c_I^\alpha=\sum
a_j c^{\alpha+1}_{I_j}$$ with $$\#(I_j)\geq\#(I)-1.$$
\end{cor}

\begin{proof}
This follows immediately from the reduction Equation \ref{ci-red}
above.
\end{proof}

\subsection{Images of basis elements in cohomology}
We would like to have, more or less, explicit formulas for the
constants $c^\alpha_I$ for $\alpha=3,2,1,0$ (the order is in terms
of difficulty).  We begin with some definitions.

\begin{definition}
Let $\Dz$ be the formal differential operator defined by
$$\Dz=\sum_{i=0}^\infty\partial_x^i.$$
\end{definition}

\begin{definition}
For $\alpha$ and $\beta$ non-negative integers, let
$\D{\alpha}{\beta}_x$ be a differential operator defined as follows:
$$\D{\alpha}{\beta}_x=\sum\partial_x...\frac{1}{x}...\frac{1}{x}...\partial_x$$
more precisely, it is the sum of all the  possible words
of length $\beta$ in the letters $\partial_x$ and $\frac{1}{x}$ with
exactly $\alpha$ of the letters being $\frac{1}{x}$.  Thus
$$\D{\alpha}{\beta}_x=0\quad\text{if}\quad\alpha>\beta.$$
\end{definition}

\begin{remark} For example $\D{0}{3}_x=\partial_x^3$,
$\D{1}{3}_x=\partial_x^2\frac{1}{x}+\partial_x \frac{1}{x}
\partial_x+\frac{1}{x}\partial_x^2$, $\D{2}{3}_x=\partial_x(\frac{1}{x})^2+\frac{1}{x}\partial_x\frac{1}{x}
+(\frac{1}{x})^2\partial_x$, and $\D{3}{3}_x=(\frac{1}{x})^3$,
etc. \end{remark}

\begin{definition}
Define the integers $\D{\alpha}{\beta}$ by
$$\D{\alpha}{\beta}=\D{\alpha}{\beta}_x x^\beta.$$  Equivalently,
$$\D{\alpha}{\beta}=\left[\left(\Dz\frac{1}{x}\right)^\alpha\Dz x^\beta\right]_0$$ where
$[g]_0$ for $g\in\cp[[z^{\pm 1}]]$ denotes its degree zero
coefficient.
\end{definition}

Notice that $\D{0}{\beta}=\beta!$ and $\D{\alpha}{\alpha}=1$.

\begin{remark}
The integers $\D{\alpha}{\beta}$ arise naturally in our computation
due to the following observation.  Define $S^{\alpha}(\beta)$
recursively by $S^0(\beta)=1$ and
$$S^\alpha(\beta)=\displaystyle\sum_{i=1}^{\beta-1}\dfrac{S^{\alpha-1}(i)}{i}.$$
Observe that this  makes sense only for $\beta>\alpha$. For
$\beta\leq\alpha$ set $S^\alpha(\beta)$ to be $0$. Then
$$\D{\alpha}{\beta-1}=(\beta-1)! S^\alpha(\beta).$$
\end{remark}

\begin{definition}
Define the formal differential operator $\D{\alpha}{}_x$ by
$$\D{\alpha}{}_x=\sum_{\beta}\D{\alpha}{\beta}_x=\left(\Dz\frac{1}{x}\right)^\alpha\Dz.$$
\end{definition}

The key property of $\D{\alpha}{\beta}$'s is demonstrated in the
following Lemma.

\begin{lemma}\label{keylemma}
The integers $\D{\alpha}{\beta}$ satisfy the relation:
$$\D{\alpha}{\beta}=\beta\cdot\D{\alpha}{\beta-1}+\D{\alpha-1}{\beta-1}$$
for $\beta>0$.  We set $\D{\alpha}{\beta}=0$ for $\alpha<0$.
\end{lemma}
\begin{proof}
For $\beta>\alpha>0$ we have
$$\D{\alpha}{\beta}=\D{\alpha}{\beta}_x x^{\beta}=\D{\alpha}{\beta-1}_x\partial_x x^\beta+\D{\alpha-1}{\beta-1}_x
\frac{1}{x}x^\beta=\beta\cdot\D{\alpha}{\beta-1}+\D{\alpha-1}{\beta-1}.$$
If $\beta=\alpha>0$ then
$\D{\alpha}{\alpha}=1=\D{\alpha-1}{\alpha-1}$ and
$\alpha\cdot\D{\alpha}{\alpha-1}=0$.  If $\beta>\alpha=0$ then
$\D{0}{\beta}=\beta!=\beta(\beta-1)!=\beta\cdot\D{0}{\beta-1}$ and
$\D{-1}{\beta-1}=0$.
\end{proof}

\begin{definition}
For $s\geq 0$ and $i\geq0$ let
$$\Ds{\alpha}{s}=(-1)^\alpha\cdot\D{\alpha}{s}$$ and
$$\Di{\alpha}{i}=\begin{cases}
\dfrac{1}{5^{\alpha+1}}\cdot\D{\alpha}{i-1} & i>0\\
\delta_{\alpha, 0} & i=0.\end{cases}$$
\end{definition}

Note that $\Di{\alpha}{1}=\dfrac{1}{5}\delta_{\alpha,
0}=\dfrac{1}{5}\Di{\alpha}{0}$.

The following is an immediate corollary of Lemma \ref{keylemma} and
is the main ingredient in the proof of Theorem \ref{huge}.
\begin{cor} For $s\geq 0$ and $i>0$,
\begin{equation}
\Di{\alpha}{i+1}=i\cdot\Di{\alpha}{i}+\frac{1}{5}\cdot\Di{\alpha-1}{i}
\end{equation}
and
\begin{equation}
\Ds{\alpha}{s+1}=(s+1)\cdot\Ds{\alpha}{s}-\Ds{\alpha-1}{s}.
\end{equation}
\end{cor}

Theorems \ref{pre-huge} and \ref{huge} describe the coefficients
$c^\alpha_I$ in a way that will be useful to us.

\begin{theorem}\label{pre-huge}
Most $c^\alpha_I$s vanish.  More precisely,
$$c^\alpha_I=0\quad\text{for}\quad\alpha+\#(I)>3.$$
\end{theorem}
\begin{proof}
Recall that $c^{\alpha>3}_I=0$ so the theorem holds in these cases.
We proceed by induction on $\alpha$.  Let $c_I^\alpha$ be such that
$\alpha+\#(I)>3$.  But by Corollary \ref{reductionconsequence},
$$c_I^\alpha=\sum a_j c^{\alpha+1}_{I_j}$$ and
$\#(I_j)+\alpha+1\geq\#(I)-1+\alpha+1>3$ so that
$c^{\alpha+1}_{I_j}$s and thus $c_I^\alpha$ vanish.
\end{proof}

\begin{definition}
Define integers $\chi=\chi_{\alpha,I}$ by
$$\chi_{\alpha,I}=3-\#(I)-\alpha.$$  Thus $c^\alpha_I=0$ if
$\chi_{\alpha,I}<0$.
\end{definition}

\begin{theorem}\label{huge}
The non-vanishing $c^\alpha_I$s are explicitly described below.
$$
\xymatrix@=.5pt{
c_s^3=\Ds{0}{s} & & & \\
c_s^2=\Ds{1}{s} & c_{is}^2=\Ds{0}{s}\cdot\Di{0}{i} & &\\
c_s^1=\Ds{2}{s} &
c_{is}^1=\!\!\!\!\displaystyle\sum_{\alpha+\beta=1}\!\!\!\!\Ds{\alpha}{s}\cdot\Di{\beta}{i}
& c_{ijs}^1
=\Ds{0}{s}\cdot\Di{0}{i}\cdot\Di{0}{j} & \\
c_s^0=\Ds{3}{s} &
c_{is}^0=\!\!\!\!\displaystyle\sum_{\alpha+\beta=2}\!\!\!\!\Ds{\alpha}{s}\cdot\Di{\beta}{i}
& c_{ijs}^0
=\!\!\!\!\!\!\!\displaystyle\sum_{\alpha+\beta+\gamma=1}\!\!\!\!\!\!\!\Ds{\alpha}{s}\cdot\Di{\beta}{i}\cdot\Di{\gamma}{j}
& c^0_{ijks}=\Ds{0}{s}\cdot\Di{0}{i}\cdot\Di{0}{j}\cdot\Di{0}{k}}
$$  More compactly, let $\chi=\chi_{\alpha,I}$ then
\begin{equation}\label{c-equation}
c^\alpha_I=c^\alpha_{ijknms}=\!\!\!\!\!\!\!
\sum_{\gamma+\beta_1+...+\beta_5=\chi}
\!\!\!\!\!\!\!\Ds{\gamma}{s}\cdot\Di{\beta_1}{i}\cdot\Di{\beta_2}{j}\cdot\Di{\beta_3}{k}\cdot\Di{\beta_4}{n}\cdot
\Di{\beta_5}{m}.
\end{equation}
\end{theorem}
\begin{proof}
It is sufficient to show that the Equation \ref{c-equation} above
satisfies the relations and conditions of the Theorem \ref{c-rel}.

Clearly the Equation \ref{c-equation} is symmetric in $i,j,k,n,m$.
If $\alpha>3$ then $\chi<0$ and so the sum is over $\emptyset$ and
thus is $0$; we conclude that the formula holds for
$c_I^{\alpha>3}$. The coefficient $c^3_0=\Ds{0}{0}=1$ as expected.

Let us verify Equation \ref{cs-red}.  Consider $s\geq 0$, note that
$\#(\{i,j,k,n,m,s+1\})=\#(\{i,j,k,n,m,s\})$.  Then
\begin{align*}
c^\alpha_{ijknm(s+1)}&=\!\!\!\!\!\!\!
\sum_{\gamma+\beta_1+...+\beta_5=\chi}
\!\!\!\!\!\!\!\Ds{\gamma}{s+1}\cdot\Di{\beta_1}{i}...
\Di{\beta_5}{m}\\
&=(s+1)\!\!\!\!\!\!\! \sum_{\gamma+\beta_1+...+\beta_5=\chi}
\!\!\!\!\!\!\!\Ds{\gamma}{s}\cdot\Di{\beta_1}{i}...
\Di{\beta_5}{m}-\!\!\!\!\!\!\!
\sum_{\gamma+\beta_1+...+\beta_5=\chi}
\!\!\!\!\!\!\!\Ds{\gamma-1}{s}\cdot\Di{\beta_1}{i}...
\Di{\beta_5}{m}\\
&=(s+1)c^\alpha_{ijknms}-\!\!\!\!\!\!\!
\sum_{\gamma+\beta_1+...+\beta_5=\chi-1}
\!\!\!\!\!\!\!\Ds{\gamma}{s}\cdot\Di{\beta_1}{i}...
\Di{\beta_5}{m}\\
&=(s+1)c^\alpha_{ijknms}-c^{\alpha+1}_{ijknms}.
\end{align*}

We verify Equation \ref{ci-red} in two steps.

First, assume that $i>0$, and so
$\#(\{i+1,j,k,n,m,s\})=\#(\{i,j,k,n,m,s\})$ and
\begin{align*}
c^\alpha_{(i+1)jknms}&=\!\!\!\!\!\!\!
\sum_{\gamma+\beta_1+...+\beta_5=\chi}
\!\!\!\!\!\!\!\Ds{\gamma}{s}\cdot\Di{\beta_1}{i+1}...
\Di{\beta_5}{m}\\
&=i\!\!\!\!\!\!\! \sum_{\gamma+\beta_1+...+\beta_5=\chi}
\!\!\!\!\!\!\!\Ds{\gamma}{s}\cdot\Di{\beta_1}{i}...
\Di{\beta_5}{m}+\dfrac{1}{5}\!\!\!\!\!\!\!
\sum_{\gamma+\beta_1+...+\beta_5=\chi}
\!\!\!\!\!\!\!\Ds{\gamma}{s}\cdot\Di{\beta_1-1}{i}...
\Di{\beta_5}{m}\\
&=i c^\alpha_{ijknms}+\dfrac{1}{5}\!\!\!\!\!\!\!
\sum_{\gamma+\beta_1+...+\beta_5=\chi-1}
\!\!\!\!\!\!\!\Ds{\gamma}{s}\cdot\Di{\beta_1}{i}...
\Di{\beta_5}{m}\\
&=i c^\alpha_{ijknms}+\frac{1}{5}c^{\alpha+1}_{ijknms}.
\end{align*}
Let $i=0$, note that
$\chi_{\alpha,\{1,j,k,n,m,s\}}=\chi_{\alpha+1,\{0,j,k,n,m,s\}}$, so
that
\begin{align*}
c^\alpha_{1jknms}&=\!\!\!\!\!\!\!
\sum_{\gamma+\beta_1+...+\beta_5=\chi}
\!\!\!\!\!\!\!\Ds{\gamma}{s}\cdot\Di{\beta_1}{1}...
\Di{\beta_5}{m}\\
&=\!\!\!\!\!\!\! \sum_{\gamma+\beta_1+...+\beta_5=\chi}
\!\!\!\!\!\!\!\Ds{\gamma}{s}\cdot\dfrac{1}{5}\delta_{\beta_1,0}...
\Di{\beta_5}{m}\\
&=\dfrac{1}{5}\!\!\!\!\!\!\! \sum_{\gamma+\beta_1+...+\beta_5=\chi}
\!\!\!\!\!\!\!\Ds{\gamma}{s}\cdot\Di{\beta_1}{0}...
\Di{\beta_5}{m}\\
&=\dfrac{1}{5}c^{\alpha+1}_{0jknms}.
\end{align*}
\end{proof}

\section{Calculation of the Frobenius matrix elements}
We are interested primarily in the first row of the Frobenius
matrix. In fact all the other entries can be deduced from general
considerations using the first row.  Namely, the matrix elements are
determined by the data of
$$(Fr\,\delta^i\omega,\delta^j\omega)_0$$ for $0\leq i,j\leq 3$.
However the compatibility of the Frobenius map and the symplectic
structure with the Gauss-Manin connection implies that
$$(Fr\,\delta^i\omega,\delta^j\omega)_0=(\frac{1}{p^i}\delta^i Fr\,\omega,\delta^j\omega)_0=
\frac{(-1)^i}{p^i}(Fr\,\omega,\delta^{i+j}\omega)_0.$$  In
particular the first column, i.e.,
$(Fr\,\delta^i\omega,\delta^3\omega)_0$, is $0$ except for the first
entry.

Unfortunately, but not surprisingly, it is the first column that is
very easy to compute directly. Then things get progressively harder.
We focus on going as far as possible with the direct computation of
the first row.  Then we provide some additional formulas for the
reader interested in computing other matrix elements directly; we hope that the resulting identities involving the coefficients of the Dwork exponential will prove interesting.

\begin{definition}
We will call the power series $$f(x)=\exp(x^p/p+x)=:\sum B_i x^i$$
the Dwork exponential.
\end{definition}

\begin{remark}
Recall that the power series $A(z)=\exp(\pi(z^p-z))$ was used to
define the action of the Frobenius.  Notice that $$A(z)=f(-\pi z).$$
\end{remark}

\begin{lemma}
The Dwork exponential coefficients $B_i$'s and $c^\alpha_I$'s can be
used to compute the first row, i.e.,
$$(Fr\,\omega,\delta^\alpha\omega)_0=-p^5 Y\sum B_i B_j B_k B_n B_m
B_s c^\alpha_{ijknm(s+p-1)}.$$
\end{lemma}
\begin{proof}
Recall from the proof of Lemma \ref{frobandh} that $Fr(\omega)$ can
be written explicitly in terms of the $\omega_I$s.  More
precisely,\begin{align*} (Fr\,\omega,\delta^\alpha\omega)_0&=p^6\sum
A_i ... A_m A_s (\omega_{ijknm(s+p-1)},\delta^\alpha\omega)_0\\
&=\frac{p^6 Y}{\pi^{p-1}}\sum A_i \frac{(-1)^i}{\pi^i} ... A_s
\frac{(-1)^s}{\pi^s}c^\alpha_{ijknm(s+p-1)}\\
&=-p^5 Y\sum B_i B_j B_k B_n B_m B_s c^\alpha_{ijknm(s+p-1)}.
\end{align*}
\end{proof}

Similarly, it is easy to see that
$$(Fr\,(xt)^j,\delta^\alpha\omega)_0=\frac{(-1)^{j+1}p^5 Y}{\pi^{jp}}\sum
B_i B_j B_k B_n B_m B_s c^\alpha_{ijknm(s+jp+p-1)}.$$  One can
express $\delta^i\omega$ in terms of $(xt)^j=\omega_j$ using
Equation \ref{sreduction}, or more directly we can compute the
coefficient of $\omega_j$ in $\delta^i\omega=\sum_{j\leqslant
i}a_j\omega_j$ as follows.  Let $T=-\delta$,
$\gamma_s=(-\pi)^s\omega_s$, then
$T\gamma_s=(s+1)\gamma_s-\gamma_{s+1}$.  It is now easy to
diagonalize $T$ so that $T^i\gamma_0=(-1)^i\delta^i\omega$ can be
computed explicitly. More precisely, if
$$\delta^i\omega=\sum_{j\leqslant i}a^i_j\omega_j$$ then $$a^i_j=(-1)^{i+j}\pi^j\sum_{\alpha+\beta=j}(-1)^{\alpha}
\frac{(\alpha+1)^i}{\alpha!\beta!}.$$

\begin{remark}
The following is an easy observation that will be essential to our
computation.  If $g(x)=\sum a_\beta x^\beta$, then $$\sum a_\beta
\cdot\D{\alpha}{\beta+k}=\left[\D{\alpha}{}_x x^k g(x)\right]_0.$$
\end{remark}

The Lemma below is inspired by \cite{katz}.  We will only need its
simplest case, namely $a=1$ when it is not difficult to check that
\begin{equation}\label{reduction}x^{p-1}f(x)=\partial_x f(x)-f(x)\end{equation} so
that\footnote{One must of course be careful of convergence issues,
however these do not pose a problem here.}
$$\left(\Dz\frac{1}{x}\right)^s x^p
f(x)=-\left(\Dz\frac{1}{x}\right)^{s-1}f(x)$$ since the Equation
\ref{reduction} implies that $\Dz x^{p-1}f(x)=-f(x).$  Thus
\begin{equation}\label{dreduction1}\D{s}{}_x x^{p-1} f(x)=-\D{s-1}{}_x \frac{1}{x}f(x),\quad s\geq
1\end{equation} and \begin{equation}\label{dreduction2}\D{0}{}_x
x^{p-1} f(x)=-f(x).\end{equation} We will sketch a general proof
below.

\begin{lemma}
The Dwork exponential $f$ satisfies the following equation
$$\left(\Dz\frac{1}{x}\right)^s x^{ap}
f(x)=(-1)^a\sum_{i=1}^a
p^{a-i}\cdot\D{i-1}{a-1}\left(\Dz\frac{1}{x}\right)^{s-i} f(x)$$ for
$a\geq 1$.
\end{lemma}

\begin{proof}
The idea, as in \cite{katz}, is to find a $g_a$ such that
$x^{ap-1}f(x)=\partial_x g_a(x)-g_a(x)$.  It is easy to see that
such a $g_a$ can be found, and it is of the form
$g_a(x)=H_a(x^p)f(x)$ where $H_a$ is a polynomial that one can write
down explicitly.  The rest is a very tedious calculation.  It is not
very difficult, using the explicit form of $H_a$, to express
$\left(\Dz\frac{1}{x}\right)^s x^{ap} f(x)$ as a linear combination
of the same objects with strictly smaller $s$ and $a$.  One then
proceeds by induction to reduce each $a$ to $0$.
\end{proof}

\subsection{First column}
We begin with the simplest case which is $(Fr\,
\omega,\delta^3\omega)_0$ and we compute it using the methods discussed above, in
particular Equations \ref{dreduction1} and \ref{dreduction2}.
\begin{align*}(Fr\, \omega,\delta^3\omega)_0&=-p^5Y\sum B_s
c^3_{s+p-1}=-p^5Y\sum B_s\cdot\D{0}{s+p-1}\\&=-p^5 Y\left[\D{0}{}_x
x^{p-1} f(x)\right]_0=p^5 Y[f(x)]_0=p^5 Y\end{align*}  Recall that
we are using the Dwork definition of the Frobenius map that
introduces an extra factor of $p^2$, so that in the standard
convention, the coefficient of $\omega$ in $Fr(\omega)$ is $p^3$.

\subsection{Second column}
Let us turn our attention to the last case where we can obtain an
exact answer by a direct computation.  \emph{This computation is the main motivation for the present paper as it obtains a result that, at least in the case of the mirror quintic, bypasses the theory of motives used in \cite{vologodsky}.}
\begin{align*}
(Fr\, \omega,\delta^2\omega)_0&=-p^5Y\left(\sum B_s c^2_{s+p-1}+5\sum B_i B_s c^2_{i(s+p-1)}\right)\\
&=-p^5Y\left(\sum B_s (-\D{1}{s+p-1})+5\sum B_i B_s \frac{1}{5}\cdot\D{0}{s+p-1}\cdot\D{0}{i-1}\right)\\
&=-p^5Y\left(-\left[\D{1}{}_x x^{p-1}f(x)\right]_0+\left[\D{0}{}_x x^{p-1}f(x)\right]_0\left[\D{0}{}_x \frac{1}{x}f(x)\right]_0\right)\\
&=-p^5Y\left(\left[\D{0}{}_x
\frac{1}{x}f(x)\right]_0-\left[\D{0}{}_x
\frac{1}{x}f(x)\right]_0\right)=0
\end{align*}

Note that the vanishing of this coefficient implies, by the
compatibility of the Frobenius map with the symplectic form, that
$(Fr\, \omega,\delta\omega)_0$ is also $0$, see \cite{ksv}.  However
we want to try to compute it directly in the next section.

\subsection{Third column}
By combining the result of last section with the direct computation
of this one, we obtain an interesting non-linear relation on the
coefficients of the Dwork exponential.  It would be interesting to see if this formula generalizes.  It turns out that the most obvious generalization is false (see below).
\begin{align*}
0=&(Fr\, \omega,\delta\omega)_0\\
&=-p^5Y\left(\sum B_s c^1_{s+p-1}+5\sum B_i B_s c^1_{i(s+p-1)}+\bnm{5}{2}\sum B_i B_j B_s c^1_{ij(s+p-1)}\right)\\
&=-p^5Y\Bigg(\sum B_s \cdot\D{2}{s+p-1}+5\sum B_i B_s
\left\{\frac{1}{5^2}\cdot\D{0}{s+p-1}\cdot\D{1}{i-1}-\frac{1}{5}\cdot\D{1}{s+p-1}\cdot\D{0}{i-1}\right\}\\
&\,\,\,\,\,\,\,\,\,\,\,\,\,\,\,\,\,\,\,\,\,\,\,\,\,\,\,\,\,\,\,\,+10
\sum B_i B_j B_s\frac{1}{5^2}\cdot\D{0}{i-1}\cdot\D{0}{j-1}\cdot\D{0}{s+p-1}\Bigg)\\
&=-p^5Y\Bigg(\left[\D{2}{}_x
x^{p-1}f\right]_0+\frac{1}{5}\left[\D{0}{}_x x^{p-1}f\right]_0
\left[\D{1}{}_x \frac{1}{x}f\right]_0-\left[\D{1}{}_x
x^{p-1}f\right]_0\left[\D{0}{}_x \frac{1}{x}f\right]_0\\
&\,\,\,\,\,\,\,\,\,\,\,\,\,\,\,\,\,\,\,\,\,\,\,\,\,\,\,\,\,\,\,\,+\frac{2}{5}
\left[\D{0}{}_x x^{p-1}f\right]_0\left[\D{0}{}_x \frac{1}{x}f\right]^2_0\Bigg)\\
&=-p^5Y\left(-\left[\D{1}{}_x
\frac{1}{x}f\right]_0-\frac{1}{5}\left[\D{1}{}_x
\frac{1}{x}f\right]_0+\left[\D{0}{}_x
\frac{1}{x}f\right]^2_0-\frac{2}{5}\left[\D{0}{}_x
\frac{1}{x}f\right]^2_0\right)
\end{align*}  We conclude that \begin{equation}\label{ordertwo}\left[\D{1}{}_x
\frac{1}{x}f\right]_0=\dfrac{\left[\D{0}{}_x
\frac{1}{x}f\right]^2_0}{2}\end{equation} which is the promised
quadratic relation on the coefficients of $f$.  It is natural, especially in view of the next section, to generalize the Equation (\ref{ordertwo}) by conjecturing that \begin{equation}\label{conj}\Delta_s:=\left[\D{s-1}{}_x
\frac{1}{x}f\right]_0-\frac{\left[\D{0}{}_x
\frac{1}{x}f\right]^s_0}{s!}\stackrel{?}{=}0.\end{equation}  The next section will show that \emph{this is false.}

\subsection{Fourth column}
At this point we perform the last computation, namely we look at
$(Fr(\omega),\omega)_0$.  It is impossible to derive its value from
general considerations of the kind considered in \cite{ksv}.
The vanishing of this last matrix entry is equivalent to the case $s=3$ of the \emph{false} formula (\ref{conj}) above. Note that we are able to use Equation
\ref{ordertwo} of the previous section.
\begin{align*}
&(Fr\,\omega,\omega)_0\\
&=-p^5Y\Bigg(\sum B_s c^0_{s+p-1}+5\sum B_i
B_s c^0_{i(s+p-1)}+
\bnm{5}{2}\sum B_i B_j B_s c^0_{ij(s+p-1)}\\
&+\bnm{5}{3}\sum B_i B_j B_k B_s c^0_{ijk(s+p-1)}\Bigg)\\
&=-p^5Y\Bigg(\sum B_s (-\D{3}{s+p-1})\\&
+5\sum B_i
B_s\left\{\frac{1}{5^3}\cdot\D{2}{i-1}\cdot\D{0}{s+p-1}
-\frac{1}{5^2}\cdot\D{1}{i-1}\cdot\D{1}{s+p-1}
+\frac{1}{5}\cdot\D{0}{i-1}\cdot\D{2}{s+p-1}\right\}\\
&+10\sum B_i B_j
B_s\Big\{\frac{1}{5^3}\cdot\D{1}{i-1}\cdot\D{0}{j-1}\cdot\D{0}{s+p-1}+\frac{1}{5^3}\cdot\D{0}{i-1}\cdot\D{1}{j-1}\cdot\D{0}{s+p-1}
\\
&-\frac{1}{5^2}\cdot\D{0}{i-1}\cdot\D{0}{j-1}\cdot\D{1}{s+p-1}\Big\}\\
&+10\sum B_i B_j B_k B_s \frac{1}{5^3}\cdot\D{0}{i-1}\cdot\D{0}{j-1}\cdot\D{0}{k-1}\cdot\D{0}{s+p-1}\Bigg)\\
&=-p^5Y\Bigg(-\left[\D{3}{}_x x^{p-1}f\right]_0
+\frac{1}{5^2}\left[\D{2}{}_x\frac{1}{x}f\right]_0\left[\D{0}{}_x
x^{p-1}f\right]_0-\frac{1}{5}\left[\D{1}{}_x
\frac{1}{x}f\right]_0\left[\D{1}{}_x x^{p-1}f\right]_0\\
&+\left[\D{0}{}_x \frac{1}{x}f\right]_0\left[\D{2}{}_x
x^{p-1}f\right]_0+\frac{2}{5^2}\left[\D{1}{}_x
\frac{1}{x}f\right]_0\left[\D{0}{}_x
\frac{1}{x}f\right]_0\left[\D{0}{}_x
x^{p-1}f\right]_0\\
&+\frac{2}{5^2}\left[\D{0}{}_x \frac{1}{x}f\right]_0\left[\D{1}{}_x
\frac{1}{x}f\right]_0\left[\D{0}{}_x
x^{p-1}f\right]_0-\frac{2}{5}\left[\D{0}{}_x
\frac{1}{x}f\right]^2_0\left[\D{1}{}_x
x^{p-1}f\right]_0\\
&+\frac{2}{5^2}\left[\D{0}{}_x
\frac{1}{x}f\right]^3_0\left[\D{0}{}_x
x^{p-1}f\right]_0\Bigg)\\
&=-p^5Y\Bigg(\left[\D{2}{}_x
\frac{1}{x}f\right]_0-\frac{1}{5^2}\left[\D{2}{}_x
\frac{1}{x}f\right]_0+\frac{1}{5}\left[\D{1}{}_x
\frac{1}{x}f\right]_0\left[\D{0}{}_x \frac{1}{x}f\right]_0\\
&-\left[\D{0}{}_x \frac{1}{x}f\right]_0\left[\D{1}{}_x
\frac{1}{x}f\right]_0-\frac{2}{5^2}\left[\D{0}{}_x
\frac{1}{x}f\right]_0\left[\D{1}{}_x
\frac{1}{x}f\right]_0-\frac{2}{5^2}\left[\D{1}{}_x
\frac{1}{x}f\right]_0\left[\D{0}{}_x \frac{1}{x}f\right]_0\\
&+\frac{2}{5}\left[\D{0}{}_x
\frac{1}{x}f\right]^3_0-\frac{2}{5^2}\left[\D{0}{}_x
\frac{1}{x}f\right]^3_0\Bigg)\\
&=-p^5Y\Bigg(\frac{24}{5^2}\left[\D{2}{}_x
\frac{1}{x}f\right]_0-\frac{24}{5^2}\left[\D{1}{}_x
\frac{1}{x}f\right]_0\left[\D{0}{}_x
\frac{1}{x}f\right]_0+\frac{8}{5^2}\left[\D{0}{}_x
\frac{1}{x}f\right]^3_0\Bigg)\\
&=-p^5Y\left(\frac{24}{5^2}\left[\D{2}{}_x
\frac{1}{x}f\right]_0-\frac{4}{5^2}\left[\D{0}{}_x
\frac{1}{x}f\right]^3_0\right).
\end{align*}  Recall that  $\Delta_3$ denotes the difference
$$\Delta_3=\left[\D{2}{}_x
\frac{1}{x}f\right]_0-\frac{\left[\D{0}{}_x
\frac{1}{x}f\right]^3_0}{3!}$$ we see that in the \emph{standard}
convention, the coefficient of $\delta^3\omega$ in $Fr\,\omega$ is
$$p^3 \frac{24}{5^2} \Delta_3.$$

Contrary to the case of the third column where a suggestively
similar expression vanishes, $\Delta_3$ is not $0$ as can be checked by
a computer and seems in fact to be, as is mentioned in the introduction, a rational multiple of $\zeta_p(3)$.

It is certainly possible, though outside the scope of this paper, to perform these same calculations for the case of the Calabi-Yau $5$-fold (and higher) family $$\lambda(x_0^7+x_1^7+x_2^7+x_3^7+x_4^7+x_5^7+x_6^7)+x_0 x_1 x_2 x_3 x_4 x_5 x_6=0.$$  As we have seen above, the calculations of the extra two columns will get progressively worse.  We expect that higher $\Delta_s$'s will make an appearance which will make it possible to test further conjectures regarding their value.  These conjectures will arise from certain motivic considerations and will be testable via a computer calculation similar to the case of $\Delta_3$.

\medskip
\noindent{\bf Acknowledgments.} We are indebted to A. Schwarz who persuaded us to attempt this computation and encouraged us along the way while providing much appreciated comments and advice. We would like to thank M. Kontsevich for useful discussions  and V. Vologodsky for helpful communications. Thanks are also due to P. Will for his assistance with Maple and P. Dragon for his computer proof of the non-vanishing.  We appreciate the hospitality of IH\'{E}S, MPIM Bonn and
University of Waterloo where parts of this paper were written.

\medskip
\noindent Institut des Hautes \'{E}tudes Scientifiques, Bures-sur-Yvette, France
\newline \emph{E-mail address}:
\textbf{shapiro@ihes.fr}

\end{document}